\documentclass[leqno,10.5pt]{article} 
\usepackage{amsmath, amssymb}  
\usepackage{amsthm} 
\usepackage{amscd}     
\usepackage{ulem}

\makeatletter
\newcommand*{\house}[1]{%
  \mathord{%
    \mathpalette\@house{#1}%
  }%
}
\newcommand*{\@house}[2]{%
  \dimen@=\fontdimen8 %
      \ifx#1\scriptscriptstyle\scriptscriptfont
      \else\ifx#1\scriptstyle\scriptfont
      \else\textfont\fi\fi
      3 %
  \sbox0{%
    $#1%
      \vrule width\dimen@\relax
      \overline{%
        \kern2\dimen@
        \begingroup 
          #2%
        \endgroup
        \kern2\dimen@
      }%
      \vrule width\dimen@\relax
      \mathsurround=1.5\dimen@ 
    $%
  }%
  \ht0=\dimexpr\ht0-\dimen@\relax
  \dp0=\dimexpr\dp0+2\dimen@\relax
  \vbox{%
    \kern\dimen@ 
    \copy0 %
  }%
}   

\usepackage{amsmath, amssymb}
\usepackage{amsthm} 
\usepackage{amscd} 
\usepackage{color}
\theoremstyle{plain} 
\newtheorem{theorem}{\indent\sc Theorem}[section]
\newtheorem{lemma}[theorem]{\indent\sc Lemma}

\theoremstyle{definition} 

\newtheorem{remark}[theorem]{\indent\sc Remark}
\newtheorem{example}[theorem]{\indent\sc Example}

\newcommand{\C}{\mathbb{C}}
\newcommand{\R}{\mathbb{R}}
\newcommand{\Q}{\mathbb{Q}}
\newcommand{\Z}{\mathbb{Z}}
\newcommand{\N}{\mathbb{N}}

\title{The digit exchanges in the rotational beta expansions \\ of algebraic numbers}
\author{\textsc{Hajime Kaneko} \  and \textsc{Makoto Kawashima}}
\date{} 
\begin{document}

\maketitle

\tableofcontents

\begin{abstract} 
In this article, we investigate the $\beta$-expansions of real algebraic numbers. 
In particular, we give new lower bounds for the number of digit exchanges 
in the case where $\beta$ is a Pisot or Salem number. 
Moreover, we define a new class of algebraic numbers, quasi-Pisot numbers and quasi-Salem numbers, which gives a generalization of Pisot numbers and Salem numbers. \par
Our method is applicable also to the digit expansions of complex algebraic numbers, which gives {{a new estimate}}. 
In particular, we investigate the digits of rotational beta expansion  {{considered}} by Akiyama and Caalim \cite{A-C2}
and zeta-expansion by Surer \cite{Su}, where the base is a quasi-Pisot or quasi-Salem number. 
\end{abstract} 
\section{Introduction}
Let $\beta>1$ be a real number.
In \cite{Re}, R\'{e}nyi {introduced  the representations of real numbers in base $\beta$}, so called $\beta$-expansions. 
Little is known on the digits of $\beta$-expansions of algebraic numbers. 
For instance, if ${{\beta=\,}}b\geq 2$ is an integer, then the $\beta$-expansion coincides with 
the usual base-$b$ expansion. Borel \cite{Bo1} conjectured 
that all algebraic irrational numbers are normal numbers in base-$b$. 
{However, if $b\geq 3$,} then it is still unknown whether the digit 1 appears infinitely many times in the base-$b$ expansions of algebraic irrational numbers. 
 In this article, we investigate the complexity of the digit expansions of real and
complex algebraic numbers.
In particular, we consider the digits of $\beta$-expansions in the case where $\beta$ is a Pisot or Salem number. 
We now recall the definition of Pisot and Salem numbers. {{Let $\beta$ be an algebraic integer}}. 
We call $\beta$ a Pisot number (resp. Salem number) if {{its conjugates over $\Q$,}} except $\beta$ itself have moduli less than $1$ (resp. if {{its conjugates over $\Q$,}} except $\beta$ itself have absolute values not greater than $1$ and there exists a conjugate of $\beta$ with absolute value $1$).

We introduce the notation throughout this article as follows.
We denote the set of nonnegative integers (resp. positive integers) by $\Z_{\ge 0}$ (resp. $\N$). We denote the integral and fractional parts of a real number $x$ by $\lfloor x \rfloor$ and $\{x\}$, respectively. We denote by $\lceil x \rceil$ the minimal integer not less than $x$ and use the Landau symbol $O$ and the Vinogradov symbols $\gg , \ll$ with their usual meaning. 
We denote the algebraic closure of the rational number field by $\overline{\Q}$ and 
{fix an embedding $\overline{\Q}\hookrightarrow \C$.}
For an algebraic number $\beta$, we denote the conjugates of $\beta$ by $\beta_i$ for $1\le i \le [\Q(\beta):\Q]$ with $\beta_1=\beta$. {We assume that $\beta_2$ is the complex conjugate of $\beta$ if $\beta\notin \R$.} 
Moreover, let $\overline{\Z}\hspace{0.7mm}(\subset \overline{\Q})$ be the set of algebraic integers. 

For an algebraic number $\alpha$, we denote {by $\house{\alpha}$ the value} 
$\max_{\sigma}|\sigma(\alpha)|$, where $\sigma$ runs thorough {the embeddings} of $\Q(\alpha)$ to $\C$.

Let $\beta>1$ be a real number. 
The $\beta$-transformation $T_{\beta}:[0,1]\longrightarrow [0,1)$ is defined by $$T_{\beta}(x):=\{\beta x\},$$
for $x\in [0,1]$. Let $\xi$ be a real number with $0\le \xi \le 1$. If $\beta=b\in \Z$, we also assume $\xi<1$. 
For $n\in \N$, we put $t_n(\beta; \xi):=\lfloor \beta T^{n-1}_{\beta}(\xi) \rfloor$.
Then we have $t_n(\beta; \xi)\in \Z \cap [0,\beta).$ The $\beta$-expansion of $\xi$ is defined by  
\begin{align*}
\xi=\sum_{n=1}^{\infty}t_n(\beta; \xi)\beta^{-n}.
\end{align*}
In the case where $\xi$ is a general nonnegative real number, then using a suitable integer $R\geq 0$ with $\beta^{-R}\in [0,1)$, we define the $\beta$-expansion of $\xi$ by 
\begin{align}\label{genbeta}
\xi=\sum_{n=1-R}^{\infty}t_n(\beta;\xi) \beta^{-n}:=\beta^R \sum_{n=1}^{\infty}t_n(\beta;\beta^{-R}\xi)\beta^{-n}.
\end{align}
{{Note that the choice of $R$ is not unique (although its choice does not affect the subsequent digit asymptotics).}}
For a positive integer $N$, 
the number of digit exchanges $\gamma(\beta,\xi;N)$ and the number of nonzero digits $\nu(\beta,\xi;N)$ are defined by
\begin{align*}
&\gamma(\beta,\xi;N):={\rm{Card}}\{n\in \N\mid n\le N, t_n(\beta;\xi)\neq t_{n+1}(\beta;\xi)\},\\
&\nu(\beta,\xi;N):={\rm{Card}}\{n\in \N\mid n\le N, t_n(\beta;\xi)\neq 0\},
\end{align*}
respectively, where Card denotes the cardinality. 
It is easily seen that we have the following relations among $\gamma(\beta,\xi;N)$ and $\nu(\beta,\xi;N)$:
\begin{align} \label{relation gamma nu}
\nu(\beta,\xi;N)\ge \dfrac{1}{2}\gamma(\beta,\xi;N)+O(1).
\end{align}

\bigskip

 R\'{e}nyi \cite{Re} showed for any $\beta>1$ that there exists a unique $T_{\beta}$-invariant measure $p_{\beta}$ on 
$[0,1]$ which is absolutely continuous with respect to the Lebesgue measure on $[0,1]$. 
In particular, $p_{\beta}$ is ergodic. We recall the $\beta$-normality of $\xi\in [0,1]. $ Let $S:=\Z\cap [0,\beta)$. 
Let $1\leq k<\ell$. For $\xi\in [0,1]$, we define the finite word $w_{k,\ell}(\beta;\xi)$ by 
\[
w_{k,\ell}(\beta;\xi):=t_{k}(\beta;\xi)t_{k+1}(\beta;\xi)\cdots t_{\ell}(\beta;\xi),
\]
for any $\xi\in [0,1]$. For any word $v=v_1\cdots v_k$ of length $k$, we define the cylinder set $[v]$ by 
\[
[v]_{\beta}:=\{\xi\in [0,1]\mid w_{1,k}(\beta;\xi)=v\}.
\]
A word $v$ is called admissible if $[v]_{\beta}$ is not empty. Recall that $\xi\in [0,1]$ is $\beta$-normal if 
\[
{\lim_{N\to\infty}\frac{1}{N}{\rm{Card}}\{n\leq N\mid w_{n,n+k-1}(\beta;\xi)= v \}=p_{\beta}([v]_{\beta})}.
\]
for any admissible finite word $v$ of arbitrary length $k$. 
Adamczewski and Bugeaud \cite{A-B} introduced a hypothesis on $\beta$-normality as follows: 
Let $\beta>1$ and $\xi\in[0,1]$ be algebraic numbers. Then $\xi$ is $\beta$-normal or $\xi$ has ultimately periodic 
$\beta$-expansion. 

Suppose that $\xi$ is $\beta$-normal. 
Then the sequences $(N^{-1}\gamma(\beta,\xi;N))_{N\geq 1}$ and $(N^{-1}\nu(\beta,\xi;N))_{N\geq 1}$ converge to 
positive values. 
The lower bounds for the number of digit exchanges of algebraic numbers were studied 
in \cite{B1,B2,B-E,K1,K2}, 
which gives partial results on the  $\beta$-normality of algebraic numbers. 
In particular, Bugeaud \cite{B2} proved the following: Let $\beta$ be a Pisot or Salem number and 
 $\xi$ an algebraic number with $t_{n}(\beta;\xi)\neq t_{n+1}(\beta;\xi)$ for infinitely many $n$. Then there exist effectively computable positive numbers $C_1(\beta,\xi)$ and $C_2(\beta,\xi)$, depending only on $\beta$ and $\xi$, such that  
\begin{align}\label{lower bound}
\gamma(\beta,\xi;N)\ge C_1(\beta,\xi)\dfrac{({\rm{log}}N)^{3/2}}{({\rm{log}}{\rm{log}}N)^{1/2}},
\end{align}
 for any $N\ge C_2(\beta,\xi)$.  
In particular, combining $(\ref{relation gamma nu})$ and $(\ref{lower bound})$, we have 
\begin{align} \label{lower bound 2}
\nu(\beta,\xi;N)\ge \dfrac{C_1(\beta,\xi)}{3}\dfrac{({\rm{log}}N)^{3/2}}{({\rm{log}}{\rm{log}}N)^{1/2}},
\end{align} 
for any sufficiently large $N$.
Lower bound $(\ref{lower bound 2})$ was improved in $\cite{K3}$ and $\cite{K4}$ as follows:
\begin{theorem} \label{Kaneko 1}  {\rm{\cite[Theorem $2.2$]{K4}}} 
Let $\beta$ be a Pisot or Salem number and $\xi$ an algebraic number with $[\Q(\beta,\xi):\Q(\beta)]=D$.
Suppose there exists a sequence $\bold{t}=(t_n)_{n\in \Z_{\ge{{1}}}}$ of integers satisfying the following two assumptions$:$

$({\rm{i}})$ There exists a positive integer $B$ such that, for any $n\in \Z_{\ge{{1}}}$,
$$0\le t_n\le B.$$
Moreover, there exist infinitely many $n$ such that $t_n>0$.
 
$({\rm{ii}})$ We have $$\xi=\sum_{n={{1}}}^{\infty}t_n\beta^{-n}.$$
Then there exist effectively computable positive constants $C_{3}=C_{3}(\beta,\xi,B)$ and $C_{4}=C_{4}(\beta,\xi,B)$, depending only on $\beta,\xi$ and $B$, such that, for any integer $N$ with $N\ge C_{4}$,
\begin{align*}
\lambda(\Gamma(\bold{t};N))\ge C_{3}\dfrac{N^{1/D}}{({\rm{log}}N)^{1/D}},
\end{align*}
where $\Gamma(\bold{t}):=\{n\in \Z_{\ge{{1}}}\mid t_n\neq 0\}$ and $\lambda(\Gamma(\bold{t};N)):={\rm{Card}}([{{1}},N]\cap \Gamma(\bold{t})).$
\end{theorem}
We note that the theorem above is also applicable to general representations of algebraic real numbers $\xi$ 
by infinite series in base-$\beta$. 

It is natural to conjecture that {{a counterpart}} of Theorem \ref{Kaneko 1} holds also for the number of digit exchanges 
{in beta expansion.} 
Consider the case where $\beta=b$ is a integer greater than 1. If the minimal polynomial of algebraic irrational $\xi$ satisfies certain assumptions, then it is known for any sufficiently large $N$ that $\gamma(\beta,\xi;N)\gg N^{1/d}$, 
where $d=[\Q(\xi):\Q]$ (see \cite{K1} and \cite{K2}). 

The main results of this article give  {{a counterpart}} of Theorem \ref{Kaneko 1} for more 
general Pisot and Salem numbers $\beta$. Moreover, our method is also applicable to  a broader class of algebraic numbers, {{that we call}} quasi-Pisot numbers and quasi-Salem numbers, which we define in Section 2. {Thus, our main results also give} new lower bounds for the number of {digit exchanges and the number of} nonzero digits {of more general numerical representation.} 
{In fact,} we also consider asymptotic behaviour of the digits in negative beta expansion and 
{rotational beta} expansion  in Section 3. 
We prove our {main results} in Section 4. 
\section{Main results}
To state our {main results}, we introduce quasi-Pisot and quasi-Salem numbers as follows:
For a complex number $z$, we denote its complex conjugate by $\overline{z}$.
Let $\beta$ be an algebraic integer with $|\beta|>1$.
We say $\beta$ is a quasi-Pisot number (resp. quasi-Salem number) if $|\beta_i|<1$ 
{for any $\beta_i\notin \{\beta,\overline{\beta}\}$} 
(resp. $|\beta_i|\le 1$ {for any $\beta_i \notin\{\beta,\overline{\beta}\}$} and there exists $1\le j \le [\Q(\beta):\Q]$ satisfying $|\beta_j|=1$).
For instance, any rational integer $b$ with $|b|\ge 2$ is a quasi-Pisot number. Any quadratic algebraic integer $\beta$ with $|\beta|>1$ and $\beta\notin \R$ is a quasi-Pisot number. If $\beta$ is a negative real number such that 
$-\beta$ is a Pisot number (resp. Salem number), then $\beta$ is a quasi-Pisot number (resp. quasi-Salem number). 
{See also example \ref{exa:cyc} for another example of quasi-Pisot numbers.} 
For examples of complex quasi-Pisot and quasi-Salem numbers, see \cite{D-J-S}  {{and \cite[Tables $6.3$ and $6.4$]{H-J}}}. 
For instance, two zeros $\beta,\overline{\beta}$ of $X^8-X^7+X^6-X^4+X^2-X+1$ with $|\beta|>1$ are 
quasi-Salem numbers. 

We give lower bounds for the digit exchanges in the representations of complex algebraic numbers 
by infinite series in base-$\beta$ in the case where $\beta$ is a quasi-Pisot or quasi-Salem number. 


\begin{theorem} \label{Kaneko}
Let $\beta$ be a quasi-Pisot or quasi-Salem number and $\xi$ an algebraic number with $D=[\Q(\beta,\xi):\Q(\beta)]$.
Let $S$ be a finite subset of $\Z[\beta]$ with $0\in S$.
Moreover, if $S\not\subset \Z$, then suppose that
$\beta\not\in\R$ and there exists {{an imaginary}} quadratic algebraic integer $\alpha\in \Q(\beta)$
such that $S$ is a finite subset of the ring of integers of $\Q(\alpha)$.

Let $R\geq 0$ and $\bold{t}=(t_n)_{n\ge 1-R}$ be a sequence of elements of $S$ satisfying $\xi={\displaystyle{\sum_{n=1-R}^{\infty}}}t_n\beta^{-n}.$
Assume there exist $\pi,A_0,A_1,\ldots,A_D\in \Z[\beta]$ with $\pi\neq 0, A_D\neq 0$ 
{satisfying the following$:$} 
\begin{align*}
&({\rm{i}}) \ A_D \xi^D+A_{D-1} \xi^{D-1}+\cdots+A_0=0, \\
&({\rm{ii}}) \ \dfrac{(\beta-1)^{D-k}A_k}{\pi} \in \Z[\beta] \ \text{for} \ 1 \le k \le D,\\ 
&({\rm{iii}}) \ \dfrac{(\beta-1)^{D}\beta^n A_0}{\pi} \notin \Z[\beta] \ \text{for all} \ n\in \N.
\end{align*}
Then there exist effectively computable {positive numbers $C_5$ and $C_6$} 
such that 
\begin{align*}
\gamma(\bold{t};N):={\rm{Card}}\{n\in \N\mid 
 n\le N, t_n\neq t_{n+1}\}\ge 
 {C_5\left(\dfrac{N}{{\rm{log}}N}\right)^{1/D} },
\end{align*}
for all $N\ge C_6$.
\end{theorem}
\begin{remark}\label{remfirst}
Let $\xi$ be an algebraic number such that $\sum_{n=0}^D A_n \xi^n=0$, where $A_0,\ldots,A_{D}\in \Z[\beta]$ and $A_D\ne 0$. Assume that ({\rm{i}}), ({\rm{ii}}), and ({\rm{iii}}) in Theorem \ref{Kaneko} holds with some $\pi\in \Z[\beta]\backslash\{0\}$. 
Let $\rho$ be any element in $\Z[\beta]$. Then $\xi+\rho$ also satisfies the same assumptions. In fact, 
putting 
$P(X)=\sum_{n=0}^D \widetilde{A_n}X^n:=\sum_{n=0}^{D}A_n(X-\rho)^n,$ we see that $P(\xi+\rho)=0$. 
Moreover, $\pi, \widetilde{A_0},\ldots,\widetilde{A_D}$ fulfill ({\rm{i}}), ({\rm{ii}}), and ({\rm{iii}}). 
 Hence, it suffices to prove Theorem \ref{Kaneko} in the case of $R=0$, 
by considering $\xi+\rho$ with $\rho=-\sum_{n=1-R}^0 t_n \beta^{-n}\in\Z[\beta]$ when $R\geq 1$.  
\end{remark}
\begin{remark}
{\textcolor{black}{Suppose that the assumption on Theorem $\ref{Kaneko}$ holds and $R=0$. }}
Then we have ${\rm{Card}}\{n\in \N \ | \ t_n\neq t_{n+1}\}=\infty.$
{Suppose on the contrary that} we have ${\rm{Card}}\{n\in \N \ | \ t_n\neq t_{n+1}\}<\infty.$ 
Then there exist $t\in S$ and $N_1\in \Z_{\ge 0}$ satisfying 
$t_n=t$ for all $n>N_1$.
Then we have 
\begin{align}
{\xi=\sum_{n=1}^{N_1}t_n\beta^{-n}+\sum_{n=N_1+1}^{\infty}t\beta^{-n} 
   =\sum_{n=1}^{N_1}t_n\beta^{-n}+\dfrac{t\beta^{-N_1}}{\beta-1}.} \label{D=1}
\end{align}
{By equality} $(\ref{D=1})$, we have $D=[\Q(\beta,\xi):\Q(\beta)]=1$. 
Then by assumptions $({\rm{i}})$ and $({\rm{ii}})$, we obtain
\begin{align*}
&A_0=-A_1\xi \in \Z[\beta], \ \ \ \dfrac{A_1}{\pi}\in \Z[\beta]. 
\end{align*}
Combining $(\ref{D=1})$ and the above relations, we have
\begin{align*}
\dfrac{(\beta-1)\beta^{N_1}A_0}{\pi}&=-\dfrac{A_1}{\pi}\cdot (\beta-1)\beta^{N_1}\xi\\
                                                &=-\dfrac{A_1}{\pi}\cdot  (\beta-1)\beta^{N_1} \left(\sum_{n=1}^{N_1}t_n\beta^{-n}+\dfrac{t\beta^{-N_1}}{\beta-1}\right)\in \Z[\beta],
\end{align*}
which contradicts assumption $({\rm{iii}})$. 
\end{remark}
{{In the theorem below, we treat the case where $\eta$ is a complex number of the form $\sum_{n=0}^{\infty}s_n\beta^{-n}$ for the technical reason of the proof.}}
\begin{theorem} \label{thm3}
Let $\beta$ be a quasi-Pisot or quasi-Salem number and $\eta$ an algebraic number with $D=[\Q(\beta,\eta):\Q(\beta)]$. 
 Let $S$ be a finite subset of $\Z[\beta]$ with $0\in S$.
Moreover, if $S\not\subset \Z$, then suppose that
$\beta\not\in\R$ and there exists {{an imaginary}} quadratic algebraic integer $\alpha\in \Q(\beta)$
such that $S$ is a finite subset of the ring of integers of $\Q(\alpha)$.


Let $\bold{s}=(s_n)_{n\ge 0}$ be a sequence of elements of $S$ satisfying 
$\eta=\sum_{n=0}^{\infty}s_n\beta^{-n}$. 
Assume there exist $B_0\in \Q(\beta)$ and $B_1,\ldots,B_D\in \Z[\beta]$ with $B_D\neq 0$ 
{satisfying the following$:$} 

\begin{align}
&\sum_{k=0}^{D}B_k\eta^k=0,  \label{cond-1}\\ 
&B_0\beta^n\notin \Z[\beta] \ \text{for all} \ n\in \N.  \label{cond-3}
\end{align}
Then there exist effectively computable {positive numbers $C_{7}$ and $C_{8}$} 
such that 
\begin{align}\label{eqn-main}
\lambda(\bold{s};N):={\rm{Card}}\{n\in \Z_{\geq 0}\mid 
 n< N, s_n\neq 0\}\ge C_{7}\left(\dfrac{N}{{\rm{log}}N}\right)^{1/D},
\end{align}
for all $N\ge C_{8}$.
\end{theorem}
In Example \ref{ex 1} and examples in Section 3, the implied constants in the symbol $\gg$
are positive and effectively computable.
\begin{example} \label{ex 1}
Let $\beta$ be a Pisot or Salem number. 
Let $p$ be a prime number which is coprime to $\beta(\beta-1)$ and unramified in $\mathcal{O}_{\Q(\beta)}$. 
Let $D$ be a positive integer. 
We consider the $\beta$-expansion $\sum_{n=1}^{\infty} t_n(\beta;\xi) \beta^{-n}$ of the number $\xi:=p^{-1/D}$. 
We see that $[\Q(\beta,\xi):\Q(\beta)]=D$. 
In fact,  let $\mathcal{P}\subset \mathcal{O}_{\Q(\beta)}$ be a prime ideal over $p$ and 
$\mathcal{R}$ the local ring of $\mathcal{O}_{\Q(\beta)}$ at  $\mathcal{P}$. 
using  Eisenstein irreducibility criterion for polynomials, we get that the polynomial $X^D-p$ is irreducible in $\Q(\beta)[X]$. Thus,  $pX^D-1$ is  irreducible in $\mathcal{R}[X]$, and so irreducible in $\Q(\beta)[X]$ by 
Gauss's lemma. 
By putting $A_DX^D+\cdots+A_0:=pX^D-1$ and $\pi:=p$, assumptions $({\rm{i}})$, $({\rm{ii}})$ and $({\rm{iii}})$ in Theorem $\ref{Kaneko}$ are satisfied. Thus, we obtain 
\[
\gamma(\beta,\xi;N)\gg \left(\frac{N}{\log N}\right)^{1/D},
\]
for any sufficiently large $N$. 
\end{example}

\section{Application to negative and rotational beta expansion}

Let $(X,\mathcal{B},\mu,T)$ be an ergodic measure-preserving system on a compact metric space $X$ with 
sigma-algebra $\mathcal{B}$ of Borel sets in $X$. Then $x\in X$ is called $T$-generic if 
\[
\lim_{N\to\infty}\frac{1}{N}\sum_{n=0}^{N-1}f (T^n x)=\int_X f d\mu,
\]
for any continuous function $f$ on $X$. If $(X,\mathcal{B},\mu,T)=([0,1],\mathcal{B},p_{\beta},T_{\beta})$, then 
$x$ is $T_{\beta}$-generic if and only if $x$ is $\beta$-normal. 

Ito and Sadahiro \cite{I-S} introduced the negative beta  expansions of real numbers, which gives a numeration  system 
where the base is a negative real number. The negative beta expansion is defined in terms of the iteration of the map 
$\widetilde{T_{-\beta}}: [-\beta/(\beta+1),1/(\beta+1)]\to [-\beta/(\beta+1),1/(\beta+1))$ defined by 
\[
\widetilde{T_{-\beta}}(x)=\left\{-\beta x+\frac{\beta}{\beta+1}\right\}-\frac{\beta}{\beta+1},
\]
where $\beta>1$ is a real number. 
Applying the theorem by Li and Yorke \cite{L-Y}, Ito and Sadahiro \cite{I-S} verified that 
there exists a unique $\widetilde{T_{-\beta}}$-invariant {measure $p_{-\beta}$} which is absolutely continuous with respect to the 
Lebesgue measure, and so {$p_{-\beta}$ is} ergodic. 

We now introduce a modified negative beta expansion studied by Liao and Steiner \cite{L-S}. 
Let $T_{-\beta}: [0,1]\to (0,1]$ be defined by $T_{-\beta}(x):=1-\{\beta x\}$, 
where $\widetilde{T_{-\beta}}$ is conjugate to $T_{-\beta}$ through the conjugacy function $f(x)=(\beta+1)^{-1}-x$. 
The $(-\beta)$-expansion of $\xi\in[0,1]$ is defined as 
\[
x=\sum_{n=1}^{\infty} t_n(-\beta;\xi) (-\beta)^{-n},
\]
where $t_n(-\beta;\xi)=\lfloor \beta T_{-\beta}^{n-1}(\xi)\rfloor+1\in \Z\cap[1,1+\beta]$. 
In the case where $\xi$ is a general real number, using a suitable integer $R\geq 0$ with $(-\beta)^{-R}\in [0,1]$, 
we define the $(-\beta)$-expansion of $\xi$ in the same way as (\ref{genbeta}). 
{{Note that the choice of $R$ is not unique (although its choice does not affect the subsequent digit asymptotics).}}
For more general numeration systems of real numbers related to beta expansion, 
see for instance  \cite{F1,G1}.
As {{a counterpart}} of the hypothesis on $\beta$-normality stated in Section 1, 
it is natural to conjecture that if $\beta>1$ and {$\xi\in \R$ are} algebraic numbers, then $\xi$ is $T_{-\beta}$-generic 
or the {$(-\beta)$-expansion} of $\xi$ is ultimately periodic. 
We consider the number of digit  exchanges $\gamma(-\beta,\xi;N)$ defined by 
\[\gamma(-\beta,\xi;N):={\rm{Card}}\{n\in \N\mid n\le N, t_n(-\beta;\xi)\neq t_{n+1}(\beta;\xi)\}.\]
\begin{example}
Let $\beta$ be a Pisot or Salem number. 
Let $p$ be a prime number which is coprime to $\beta(\beta-1)$ and unramified in $\mathcal{O}_{\Q(\beta)}$. 
Let $D$ be a positive integer. 
We consider the  $(-\beta)$-expansion $\sum_{n=1}^{\infty} t_n(\beta;\xi) \beta^{-n}$ of $\xi$, where 
$\xi$ is a unique zero of the polynomial $pX^D+pX^{D-1}+\cdots+pX-1$ with $0<\xi<1$. In the same way as Example \ref{ex 1}, we see that 
\[
\gamma(-\beta,\xi;N)\gg \left(\frac{N}{\log N}\right)^{1/D},
\]
for any sufficiently large $N$. 
\end{example}

Akiyama and Caalim \cite{A-C2} 
defined a rotational beta expansion, which is  a natural generalization of beta expansion for the complex plane. 
We introduce a special version of this expansion. 
Let $\beta$ be a complex number with $\beta\not\in \R$ and $|\beta|>1$. 
Let $\tau_1, \tau_2 \in \C\backslash \{0\}$ with $\tau_1/\tau_2\not\in \R$. 
Denote $F:=\{ x\tau_1+y\tau_2\mid x\in [-1/2,1/2) , \ y\in[-1/2,1/2)\}$. Let $\overline{F}$ be the closure of $F$. 
Define a map $T=T_{\beta,\tau_1,\tau_2}:\overline{F}\to F$ by 
\begin{align}\label{defdelta}
T(z):=\beta z-\delta(z),
\end{align}
where $\delta(z)$ is {a unique element} in $\Z \tau_1+\Z\tau_2$ satisfying $\beta z-\delta(z)\in F$. 
We denote by $S=S(\beta,\tau_1,\tau_2)$ the set of digits 
$\delta(z)$ with $z\in \overline{F}$. 
{Note that $S$ is a finite set.} 
Then the rotational $\beta$ expansion, or simply $\beta$-expansion, of $\xi\in F$ is defined by 
\[\xi=\sum_{n=1}^{\infty} d_n \beta^{-n},\]
where $d_n=d_n(\beta,\tau_1,\tau_2;\xi)=\delta(\beta T^{n-1}(z))\in S$. 
In the case where $\xi$ is a general complex number, using a suitable nonnegative integer $R$ with $\beta^{-R}\xi\in F$, 
we define the rotational $\beta$ expansion of $\xi$ in the same way as (\ref{genbeta}). 
We define the number of digit  exchanges in the rotational beta expansion of $\xi$ by 
\[\gamma(\xi;N)=\gamma(\beta,\tau_1,\tau_2,\xi;N):={\rm{Card}}\{n\in \N\mid n\le N, 
d_n(\beta,\tau_1,\tau_2;\xi) \neq d_{n+1}(\beta,\tau_1,\tau_2;\xi)\}.\]

Akiyama and Caalim \cite{A-C2} gave a sufficient condition for $\beta,\tau_1,\tau_2$ {{which guarantee the uniqueness of absolutely continuous}} 
invariant probability measure {$p_{\beta}=p_{\beta,\tau_1,\tau_2}$} 
on $F$, and $p_{\beta}$ is equivalent to 
the Lebesgue measure {on $F$.} 

{Surer \cite{Su} also investigated a numerical system of complex numbers called zeta-expansion.} We introduce a special version of this numerical system. 
Let again $\beta$ be a complex number with $\beta\not\in \R$ and $|\beta|>1$. 
Set $\tau_1:=1$ and $\tau_2:=-\overline{\beta}$. 
Then we have $F=\{x-y\overline{\beta}\mid x\in [-1/2,1/2) , \ y\in[-1/2,1/2)\}$ and 
$\beta F=\{{{-|\beta|^2 y+\beta x}}\mid x\in [-1/2,1/2) , \ y\in[-1/2,1/2)\}$. 
It is remarkable that if ${{z}}\in F$, then we have $\delta(z)\in \Z$ because the imaginary parts of $-\overline{\beta}$ and $\beta$ coincide, where $\delta(z)$ is defined by (\ref{defdelta}). Then the zeta-expansion of $\xi\in \C$ is defined by 
\[\xi=\sum_{n=1-R}^{\infty} d_n(\beta,1,-\overline{\beta};\xi) \beta^{-n}.\]

\begin{example}
Let $\beta$ be a quasi-Pisot or quasi-Salem number. 
Let $p$ be a prime number which is coprime to $\beta(\beta-1)$ and unramified in $\mathcal{O}_{\Q(\beta)}$. 
Let $D$ be a positive integer. 
We consider the {zeta-expansion of $\xi$}, where 
$\xi$ is a zero of the polynomial $pX^D+pX^{D-1}+\cdots+pX-1$. In the same way as Example \ref{ex 1}, we see that 
\[
\gamma(\beta,1,-\overline{\beta};\xi)\gg \left(\frac{N}{\log N}\right)^{1/D},
\]
for any sufficiently large $N$  
because the digits of zeta-expansions are rational integers. 
\end{example}


We give an example of the digit exchanges for a rotational beta expansion whose digits are not generally rational integers.

\begin{example}\label{exa:cyc}
Let $\zeta_7:=e^{2\pi i/7}$ be  the primitive $7$-th root of unity. 
For an integer $a\ge 2$, we put $$\xi_a:=\zeta^{(1-a)/2}_7\dfrac{\zeta^a_7-1}{\zeta_7-1}.$$ Then we have $\xi_a=\pm\tfrac{\sin(\pi a/7)}{\sin(\pi/7)}\in \R$.
Define  the multiplicative group $C:=\{\zeta^{m_1}_7\xi^{m_2}_2\xi^{m_3}_3\mid m_1,m_2,m_3\in \Z\}$. Note that $C$ is called the group of cyclotomic units of $\Q(\zeta_7)$ and 
 $C$ is a finite index subgroup of the units group of $\Z[\zeta_7]$ (see Section $8$ in \cite{W}).  Let $m_1,m_2,m_3$ be positive integers with $1\le m_1 \le 6$ and 
\begin{align} \label{m1m2}
m_2{\rm{log}}(1.247)+m_3{\rm{log}}(0.554)<0.
\end{align}
Put  $\beta:=\zeta^{m_1}_7\xi^{m_2}_2\xi^{m_3}_{{{3}}}\in C$. 
Let $\sigma_i\in {\rm{Gal}}(\Q(\zeta_7)/\Q)$ with $\sigma_i(\zeta_7)=\zeta^i_7$ for $1\le i \le 6$. Remark $\sigma_i(\xi_a)=\xi_{ia}\xi^{-1}_i$ and 
\begin{align*}
&|\xi_2|=1.8019377358\ldots \enspace, \ \ |\xi_3|=2.24697960372\ldots \enspace, \\
&|\sigma_2(\xi_2)|=1.24697960372\ldots \enspace, \ \ |\sigma_2(\xi_3)|=0.55490813208\ldots \enspace,\\
&|\sigma_3(\xi_2)|=0.445041867911\ldots \enspace, \ \ |\sigma_3(\xi_3)|=0.8019377358\ldots \enspace.
\end{align*}
By {{the}} above equalities and $(\ref{m1m2})$, we obtain 
$$|\beta|=|\sigma_6(\beta)|>1, \ \  |\sigma_3(\beta)|=|\sigma_4(\beta)|<|\sigma_2(\beta)|=|\sigma_5(\beta)|<1.$$
Thus the number $\beta$ is a quasi-Pisot number and $\Q(\zeta_7)=\Q(\beta)$. 
 Using the Legendre symbol $\left(\tfrac{\cdot}{\, 7 \,}\right)$, we put $\alpha:={\displaystyle{\sum_{a=1}^7}}\left(\tfrac{a}{\, 7 \,}\right)\zeta^a_7$ ({{a}} Gauss sum). Then $\alpha$ is an imaginary quadratic integer. 
Let $\mathcal{O}_{\Q(\alpha)}$ {(resp. $\mathcal{O}_{\Q(\zeta_7)}$)} be the ring of integers of $\Q(\alpha)$ {(resp. $\Q(\zeta_7)$)}. 
We consider rotational {$\beta$-expansion}, 
where $\tau_1,\tau_2$ are elements of {\textcolor{black}{$\mathcal{O}_{\Q(\alpha)}\cap \Z[\beta]$}} with $\tau_1/\tau_2\not\in \R$. 
Then the set $S$ of the digits in the rotational $\beta$ expansion satisfies the  assumptions of
Theorem $\ref{Kaneko}$. \par
Let $p$ be a prime number which is coprime to $\beta(\beta-1)$ and unramified in $\mathcal{O}_{\Q(\zeta_7)}$. 
Let $D$ be a positive integer. Then, by the same arguments in Example $\ref{ex 1}$, the polynomial $pX^D-1$ is irreducible in $\Q(\zeta_7)[X]$.
Put  $\xi:=p^{-1/D}$, $\pi:=p$, $A_{D}:=p,A_{D-1}=\ldots=A_1=0$ and $A_0:=-1$. 
 Then the numbers $\xi,\pi, A_D,\ldots,A_0$ satisfy assumptions $({\rm{i}})$, $({\rm{ii}})$ and $({\rm{iii}})$ in Theorem $\ref{Kaneko}$.
Hence, we obtain that 
\[
\gamma(\beta,\tau_1,\tau_2;\xi)\gg \left(\frac{N}{\log N}\right)^{1/D},
\]
for any sufficiently large $N$. 
\end{example}

Akiyama and Caalim \cite{A-C1} introduced a rotational beta expansion in $\R^m$, which is a natural generalization of 
the rotational expansion in $\C$. 
The rotational beta expansion in $\R^m$ is defined in terms of a map $T(z)=\beta M z$ for $z\in \R^m$, 
where $\beta>1$ is a real number and $M$ is an orthogonal matrix of order $m$. 
It is a future work to investigate the uniformity of the digits in the rotational beta expansion of elements of $\R^m$. 

\section{Proof of main results}
\subsection{Reduction of Theorem \ref{Kaneko} to Theorem \ref{thm3}}
 By Remark \ref{remfirst}, we may assume that $R=0$.
Firstly, we reduce Theorem $\ref{Kaneko}$ to Theorem \ref{thm3}.
Define the sequence $\bold{v}=(v(m))_{m\in \Z_{\ge 0}}$ of nonnegative integers  by $v(0)=-1+N_0$, $v(1)<v(2)<\cdots$ and $$\{n\in \N \ | \ t_n\neq t_{n+1}\}=\{v(m) \ | \ m\in \N\}.$$
Denote $t_{v(m)}(=t_{1+v(m-1)})$ by $x(m)$ for any $m\in \N$. 
Note that $x(1)=t_{N_0}\ne 0$. We see  
\begin{align*}
\xi&=\sum_{n=1}^{\infty}t_n\beta^{-n}=\sum_{m=1}^{\infty}\sum_{n=1+v(m-1)}^{v(m)}t_n\beta^{-n} \\
    &=\dfrac{1}{\beta-1}\left(x(1)\beta^{-v(0)}+\sum_{m=1}^{\infty}(x(m+1)-x(m))\beta^{-v(m)}\right)=\dfrac{1}{\beta-1}\sum_{n=0}^{\infty}s_n \beta^{-n}, 
\end{align*}
where the sequence $(s_n)_{n\in \Z_{\ge0}}$ of integers is defined by 
$$s_n=
\begin{cases}
x(1) \ & \text{if} \ n=v(0),\\
x(m+1)-x(m) \ & \text{if there exists} \ m\in \N \ \text{satisfying} \ n=v(m),\\ 
0 \ & \text{otherwise}.
\end{cases}
$$
 Putting $\widetilde{S}:=\{a-b\mid a,b\in S\}$, we see that $\beta$ and $\widetilde{S}$ satisfy the assumptions of 
Theorem \ref{thm3}. 
In fact, if $S\not\subset \Z$, then $\widetilde{S}$ is a finite subset in the ring of integers of $\Q(\alpha)$, where 
$\alpha$ is denoted in the assumption of Theorem \ref{Kaneko}. 
Note that $(s_n)_{n\in \Z_{\ge0}}$ is bounded and $s_n\in \widetilde{S}$ for any $n\geq 0$. 
Putting $\eta=(\beta-1)\xi$ and $B_k=A_k(\beta-1)^{D-k} \pi^{-1}$ for $0\le k \le D$, 
we get that 
$\eta$, $B_k$ ($0\leq k\leq D$)  satisfy (\ref{cond-1}) and (\ref{cond-3}) by assumptions $({\rm{i}})$, $({\rm{ii}})$ and $({\rm{iii}})$ of Theorem $\ref{Kaneko}$.  
Since $x(1)\ne 0$, we see $\{n\in \N \ | \ s_n\neq 0\}=\{v(0)\}\cup \{n\in \N \ | \ t_n\neq t_{n+1}\}$, and so we reduced the proof of Theorem $\ref{Kaneko}$ to the proof of Theorem \ref{thm3}. 

%
\subsection{Preliminaries for the proof of Theorem \ref{thm3}}
In what follows, the implied constants in the symbols $\gg, \ll$ and the constants $C_{9},C_{10},\ldots$ are effectively computable {positive ones}. 
If necessarily, changing $\eta \beta^{N}$ with suitable nonnegative integer $N$ by $\eta$, we may assume that $s_0\ne 0$. 
Let $\Gamma:=\{n\in {{\Z_{\ge 0}}} \ | \ s_n\neq 0\}$ and $\lambda(\Gamma;N):={\rm{Card}}\{n\in \Z_{\ge 0} \ | \ s_n\neq 0, \ n< N\}=\lambda(\bold{s};N)$ for $N\in \N$. 
For a positive integer $k$, we have 
\begin{align*}
\eta^k=\left(\sum_{n=0}^{\infty}s_n\beta^{-n}\right)^k=\sum_{m=0}^{\infty}\left(\sum_{\substack{m_1,\ldots,m_k\in \Gamma \\ m_1+\cdots+m_k=m}}s_{m_1}\cdots s_{m_k}\right) \beta^{-m}.
\end{align*} 
For $m\in \Z_{\ge 0}$, we denote the complex number $\sum_{\substack{m_1,\ldots,m_k\in \Gamma \\ m_1+\cdots+m_k=m}}s_{m_1}\cdots s_{m_k}$ by $\rho(k;m)$. 
If  $S\subset \Z$, then put $\mathcal{O}:=\Z$. 
If  $S\not\subset \Z$, then let $\mathcal{O}$ be the ring 
of integers of $\Q(\alpha)$.  
By $s_m\in S\subset \mathcal{O}\cap\Z[\beta]$ for any $m\geq 0$, we see $\rho(k;m)\in \mathcal{O}\cap\Z[\beta]$ and 
\begin{align} \label{upper bound rho}
\house{\rho(k;m)}\le T^k(m+1)^k,
\end{align}
where $T=\max\{\house{\alpha}\mid \alpha\in S\}$. 
Moreover, if $\rho(k;m)\ne 0$, then
\begin{align}\label{lower_rho}
{|\rho(k;m)|\geq C_{9}:=\inf\{|z|\mid z\in \mathcal{O}\backslash \{0\}\}>0},
\end{align}
by $\rho(k;m)\in \mathcal{O}$. 

For any integer $k$ with $0 \le k \le D$, we define the set $k\Gamma$ of integers  by
$$k\Gamma=
\begin{cases}
\{0\} \ & \text{if} \ k=0,\\
\{m_1+\cdots+m_k \ | \ m_1,\ldots, m_k\in \Gamma\} \ & \text{if} \ 1\le k \le D.
\end{cases}
$$
Note that if  $m\notin k\Gamma$ we have $\rho(k;m)=0$. Moreover, $1 \Gamma=\Gamma$, $0\in \Gamma$ and $0\Gamma\subset \Gamma\subset \cdots \subset D\Gamma.$
For a positive integer $N$, put $\lambda(k\Gamma;N)={\rm{Card}}(k\Gamma\cap [0,N))$. By the definition of $k\Gamma$, we have
\begin{align*}
{\rm{Card}}\{m\in \Z_{\ge 0} \ | \ m<N, \rho(k;m)\neq 0\} \le \lambda(k\Gamma;N)\le {\rm{Card}}(\Gamma\cap [0,N))^k=\lambda(\Gamma;N)^k.
\end{align*}
By the above equality, we shall estimate the lower bounds for ${\rm{Card}}\{m\in \Z_{\ge 0} | \ m<N, \rho(k;m)\neq 0\}$. 
 Since it is difficult to estimate the lower bounds for them directly, we consider 
the complex number $Y_R$ defined by 
\begin{align*}
Y_R=\sum_{k=1}^{D}B_k \sum_{m=1}^{\infty}\beta^{-m}\rho(k;m+R),  
\end{align*}
where $R\in \Z_{\geq 0}$. 
We prove that $Y_R\neq 0$. By equality $(\ref{cond-1})$, we have 
\begin{align*}
0=\sum_{k=0}^{D}B_k\eta^k 
  =B_0+\sum_{k=1}^{D}B_k\sum_{m=0}^{\infty}\beta^{-m}\rho(k;m). 
\end{align*}
{{Multiplying the above equality by $\beta^{R}$}}, we obtain
\begin{align*} 
0&=B_0\beta^R+\sum_{k=1}^{D}B_k\sum_{m=0}^{\infty}\beta^{-m+R}\rho(k;m)=B_0\beta^R+\sum_{k=1}^{D}B_k\sum_{m=-R}^{\infty}\beta^{-m}\rho(k;m+R).
\end{align*}
By the above equalities, 
\begin{align} \label{Y_R}
Y_R=-B_0\beta^{R}-\sum_{k=1}^{D}B_k\sum_{m=-R}^{0}\beta^{-m}\rho(k;m+R).
\end{align} 
In particular, $Y_R$ is an algebraic number. 
If $Y_R=0$ then, by $(\ref{Y_R})$ and $B_k\in \Z[\beta]$ for $1\le k \le D$, we have {{$B_0 \beta^{R}\in \Z[\beta]$}},
 which contradicts $(\ref{cond-3})$. 
Hence, we conclude $Y_R\neq 0$.
\begin{lemma} \label{lower bound of Y_R}
There exist positive integers $C_{10},C_{11}$ satisfying $|Y_R|>R^{-C_{10}}$ for all $R\ge C_{11}$.
\end{lemma}
\begin{proof}
Put $d={\rm{deg}} \ \beta$ and denote the set of embeddings of $\Q(\beta)$ into $\C$ by $\{\sigma_1,\ldots,\sigma_d\}$. We describe $\sigma_1(x)=x$ for all $x\in \Q(\beta)$ and denote the complex conjugate of $\sigma_1$ by $\sigma_2$ if $\beta\notin \R$.

Let $2\le i \le d$ if $\beta\in \R$ (resp. $3 \le i \le d$ if $\beta\notin \R$). 
 Recall that $\rho(k;m)\in  \Z[\beta]$ for any $1\leq k\leq D$ and $m\in \Z_{\geq 0}$. 
By $(\ref{Y_R})$, we have
\begin{align}
|\sigma_i(Y_R)|&\le  |\sigma_i(B_0\beta^{R})|+\sum_{k=1}^{D}|\sigma_i(B_k)|\sum_{m=-R}^{0}|\sigma_i(\beta^{-m})| |\sigma_i(\rho(k;m+R))|\nonumber \\
                   &\le  |B_0|+\sum_{k=1}^{D}|\sigma_i(B_k)|\sum_{m=-R}^{0}T^k(R+m+1)^k  \nonumber \\
                   &\ll (R+1)^{D+1} \label{upper sigma Y_R}.
\end{align} 
Note that in the above second inequality, we use $|\sigma_i(\beta)|\le 1$ and $(\ref{upper bound rho})$. 
Take a positive integer $J$ satisfying $JB_0\in \overline{\Z}$. By equality $(\ref{Y_R})$  and 
$B_k\in \Z[\beta]$ for $k=1,\ldots, D$, we have $JY_R\in \overline{\Z}$.

\


If $\beta \in \R$, then we have 
\begin{align*} 
1\le |JY_R|\prod_{i=2}^{d}|\sigma_i(JY_R)|\ll |Y_R|(R+1)^{(d-1)(D+1)},
\end{align*}
by $(\ref{upper sigma Y_R})$.
In the case of $\beta \notin \R$, also using $(\ref{upper sigma Y_R})$, we have
\begin{align*}
1\le |JY_R|^2\prod_{i=3}^{d}|\sigma_i(JY_R)|\ll |Y_R|^2(R+1)^{(d-2)(D+1)}.
\end{align*}

In both cases, there exist positive integers $C_{10},C_{11}$ satisfying $|Y_R|>R^{-C_{10}}$ for all $R\ge C_{11}$, 
which completes the proof of Lemma $\ref{lower bound of Y_R}$.
\end{proof}
Let $N$ be a positive integer. We put \[{{C_{12}=\frac{|B_D|}{2|\beta|}C_9}},
\] 
and {$y_N:={\rm{Card}}\{R\in \Z_{\ge0} \ | \ R<N, |Y_R|\ge C_{12}\}$}.
\begin{lemma} \label{upper bound y_N}
For all sufficiently large integer $N$, we have $$y_N\ll {\rm{log}}N+\lambda(\Gamma;N)^D.$$
\end{lemma}
\begin{proof}
Put $K=\lceil(D+1){\rm{log}}_{|\beta|}N \rceil$, where $\log_{|\beta|} x=(\log x)/(\log |\beta|)$. Then by the definition of $y_N$, we have
\begin{align*}
y_N&\le K+y_{N-K+1}
     =K+\sum_{\substack{0\le R \le N-K \\ |Y_R|\ge C_{12}}}1\\
     &\le K+\dfrac{1}{C_{12}}\sum_{R=0}^{N-K}|Y_R|.
\end{align*}
We estimate the upper bound for $\sum_{R=0}^{N-K}|Y_R|$. By the definition of $Y_R$, we obtain
\begin{align*}
\sum_{R=0}^{N-K}|Y_R|&\le \sum_{R=0}^{N-K}\sum_{k=1}^{D}\sum_{m=1}^{\infty}|B_k \beta^{-m}\rho(k;m+R)|\\
                              &=\sum_{k=1}^{D}|B_k|\sum_{m=1}^{\infty}\sum_{R=0}^{N-K}|\beta^{-m} \rho(k;m+R)|\\
                              &=\sum_{k=1}^{D}|B_k|z_N(k),
\end{align*}
where $z_N(k)=\sum_{m=1}^{\infty}\sum_{R=0}^{N-K}|\beta^{-m}\rho(k;m+R)|$. To obtain the assertion of Lemma $\ref{upper bound y_N}$, it is enough to prove 
\begin{align} \label{upper z k N}
z_N(k)\ll \lambda(\Gamma;N)^{D} \ \text{for all} \ 1\le k \le D.
\end{align}
Put $S_1(k)=\sum_{m=1}^{{\textcolor{black}{K-1}}}|\beta|^{-m}\sum_{R=0}^{N-K}|\rho(k;m+R)|$ and 
$S_2(k)=\sum_{m={\textcolor{black}{K}}}^{\infty}|\beta|^{-m}\sum_{R=0}^{N-K}|\rho(k;m+R)|$.
Note that by the definition of $z_N(k)$, we have $z_N(k)=S_1(k)+S_2(k)$.
First, we estimate the upper bound for $S_1(k)$ as follows: 
\begin{align} 
S_1(k)&\le \sum_{m=1}^{{\textcolor{black}{K-1}}}|\beta|^{-m}\sum_{R=0}^{{\textcolor{black}{N-1}}}|\rho(k;R)| 
        \le \sum_{k=0}^{\infty}|\beta|^{-k}\sum_{k=0}^{\textcolor{black}{N-1}}|\rho(k;R)| \nonumber \\
        &\ll \sum_{R=0}^{{\textcolor{black}{N-1}}}|\rho(k;R)| 
        \le \sum_{R=0}^{\textcolor{black}{N-1}}\sum_{\substack{m_1,\ldots,m_k\in \Gamma \\ m_1+\cdots+m_k=R}}|s_{m_1}\cdots s_{m_k}| \nonumber \\
        &=\sum_{\substack{m_1,\ldots,m_k\in \Gamma \\ {\textcolor{black}{m_1+\cdots+m_k< N}}}}|s_{m_1}\cdots s_{m_k}| 
        \le T^{k}\sum_{\substack{m_1,\ldots,m_k\in \Gamma \\ {\textcolor{black}{m_1+\cdots+m_k<N}}}}1 \label{remark 1}\\
        &\le T^k\lambda(\Gamma;N)^k 
         \ll \lambda(\Gamma;N)^D. \label{conc 1}
\end{align}
Note that $(\ref{remark 1})$ is obtained {{in a similar way as inequality $(\ref{upper bound rho})$}}.
Second, we estimate the upper bound for $S_2(k)$ as follows: 
\begin{align}
S_2(k)&\ll \sum_{m={\textcolor{black}{K}}}^{\infty}|\beta|^{-m}\sum_{R=0}^{N-K}(m+R+1)^{D} \label{use 8} \\
        &\le \sum_{m={\textcolor{black}{K}}}^{\infty}|\beta|^{-m}\sum_{R=0}^{N-K}(m+N)^D \nonumber \\
        &\le \sum_{m={\textcolor{black}{K}}}^{\infty}|\beta|^{-m}N(m+N)^D. \label{S_2 1}
\end{align}
Note that in inequality $(\ref{use 8})$, we use $(\ref{upper bound rho})$. If $N\gg 1$, we have
\begin{align} \label{ineq m}
\left(\dfrac{m+1+{\textcolor{black}{N}}}{m+N}\right)^D \le \dfrac{1+|\beta|}{2} \ \text{for any} \ m\ge 1.
\end{align}
Then combining $(\ref{S_2 1})$ and $(\ref{ineq m})$, we obtain
\begin{align*}
S_2(k)&\ll 
{\textcolor{black}{|\beta|^{-K} N(K+N)^D}}
\sum_{m=0}^{\infty}|\beta|^{-m}\left(\dfrac{1+|\beta|}{2}\right)^m 
        \ll N^{D+1}|\beta|^{-K} 
        \le 1. 
\end{align*}
Combining $(\ref{conc 1})$ and the above relation, we obtain $(\ref{upper z k N})$, which completes the proof of Lemma $\ref{upper bound y_N}$. 
\end{proof} 

\subsection{Relation between $Y_R$ and $Y_{R-1}$}
For a positive integer $N$, we put $\tau(N)=\tau:={\rm{Card}} ((D-1)\Gamma;N)$ and 
define the sequence of integers {\textcolor{black}{$(i(h))_{1\le h \le \tau+1}$}} by
$[0,N)\cap (D-1)\Gamma=:\{{\textcolor{black}{0=i(1)}}<\cdots<i(\tau)\}$ and $i(\tau+1):=N$. Note that we have 
\begin{align}\label{upper tau}
\tau(N)\le \lambda(\Gamma;N)^{D-1}.
\end{align}
For $1\le h \le \tau$, we put $I_h=[i(h),i(h+1))\cap \Z$ and $y_N(h)={\rm{Card}}\{R\in I_h \mid |Y_R|\ge C_{12}\}$. 
Then by the definition of $I_h$ and $y_N(h)$, we have
\begin{align}
&\sum_{h=1}^{\tau}{\rm{Card}}\hspace{0.7mm}I_h=N, \label{=N}\\
&\sum_{h=1}^{\tau}y_N(h)=y_N. \label{=y_N}
\end{align}
\begin{lemma} \label{relation R R-1}
Let $h,R$ be integers satisfying $1\le h \le \tau$ and $R\in (i(h),i(h+1))$. Then we have 
\begin{align*}
Y_{R-1}=\dfrac{B_D}{\beta}\rho(D;R)+\dfrac{1}{\beta}Y_R.
\end{align*}
\end{lemma} 
\begin{proof} 
Firstly, since $(i(h),i(h+1))\cap (D-1)\Gamma=\emptyset$, we have $(i(h),i(h+1))\cap k\Gamma=\emptyset$ for $1\le k \le D-1$ and
\begin{align}\label{zero}
\rho(k;R)=0 \ \text{for} \ 1\le k \le D-1.
\end{align}
By the definition of $Y_{R-1}$, we have
\begin{align*}
Y_{R-1}&=\sum_{k=1}^{D}B_k \sum_{m=1}^{\infty}\beta^{-m}\rho(k;m+R-1) \nonumber \\ 
         &=\sum_{k=1}^DB_k\beta^{-1}\rho(k;R)+\sum_{k=1}^{D}B_k \sum_{m=2}^{\infty}\beta^{-m}\rho(k;m+R-1) \nonumber \\   
         &=\dfrac{B_D}{\beta}\rho(D;R)+\sum_{k=1}^{D}B_k \sum_{m=1}^{\infty}\beta^{-m-1}\rho(k;m+R)\\ 
         &=\dfrac{B_D}{\beta}\rho(D;R)+\beta^{-1}Y_R. \nonumber 
\end{align*}
Note that in above third equality, we use equality $(\ref{zero})$, which completes the proof of Lemma $\ref{relation R R-1}$.
\end{proof}
\begin{lemma} \label{gap upper bound}
Let $N\gg 1$. Put $C_{13}=1+D+C_{10}$. Let $h$ and $R$ be integers with $1\le h \le \tau$ and 
\begin{align}\label{assumpR}
i(h)+3C_{13}{\rm{log}}_{|\beta|}N<R<i(h+1).
\end{align}
Then we have 
\begin{align} \label{assertion gap upper bound}
R-\max\{R^{\prime}\mid  R^{\prime}<R, |Y_{R^{\prime}}| \ge C_{12}\}\le 2C_{13}{\rm{log}}_{|\beta|}N.
\end{align} 
\end{lemma}
\begin{proof} 
Let $N$ be a sufficiently large integer such that (\ref{ineq m}) holds. 
 Taking $R$ with (\ref{assumpR}), we see by the definition of $Y_R$ that 
\begin{align*}
|Y_R|&\le \sum_{k=1}^{D}|B_k|\sum_{m=1}^{\infty}|\beta^{-m} \rho(k;m+R)| \nonumber \\
       &\ll \sum_{k=1}^{D}|B_k|\sum_{m=1}^{\infty}|\beta|^{-m}(m+R+1)^D \nonumber \\
       &\le \sum_{k=1}^{D}|B_k|\sum_{m=1}^{\infty}|\beta|^{-m}(m+{\textcolor{black}{N}})^D  \nonumber \\
       &\ll \sum_{m=1}^{\infty}|\beta|^{-m}(m+{\textcolor{black}{N}})^D  \nonumber \\
       &\le |\beta|^{-1}({\textcolor{black}{N+1}})^D\sum_{m=0}^{\infty}|\beta|^{-m}
       {\textcolor{black}{\left(\frac{1+|\beta|}{2}\right)^m}}.
\end{align*}
Then by {{the}} above inequalities, we have  for any $N\gg 1$ that 
\begin{align} \label{upper Y_R conclusion}
|Y_R|<N^{D+1}.  
\end{align}
Put $S=\lceil C_{13} {\rm{log}}_{|\beta|}N\rceil$. Assume that we have $$\rho(D;R-m)=0 \ \text{for all} \ 0 \le m \le S.$$
Since $i(h)<R-S<\cdots< R-1<R<i(h+1)$, using Lemma $\ref{relation R R-1}$ and the assumption above, we obtain
\begin{align*}
|\beta^{S+1} Y_{R-S-1}|=\cdots=|\beta^{2}Y_{R-2}|=|\beta Y_{R-1}|=|Y_R|<N^{D+1}.
\end{align*}
Recall that $C_{10}$ and $C_{11}$ are the positive integers defined in Lemma $\ref{lower bound of Y_R}$. By the above inequality  and Lemma $\ref{lower bound of Y_R}$, we have 
\begin{align*}
|\beta|^{S+1}<N^{D+1} |Y_{R-S-1}|^{-1}<N^{D+1}(R-S-1)^{C_{10}} \le N^{D+1+C_{10}}=N^{C_{13}}.
\end{align*}
Take $N$ satisfying $(\ref{upper Y_R conclusion})$ and $R-S-1\ge {\rm{log}}_{|\beta|}N \ge C_{11}$.
 Thus, we get 
$$S+1=\lceil  C_{13} {\rm{log}}_{|\beta|}N\rceil+1 <C_{13}{\rm{log}}_{|\beta|}N,$$
a contradiction. Hence, there exists  $m'$ with $0 \le m' \le S$ satisfying 
$\rho(D;R-m')\ne 0$, and so 
\begin{align}\label{lowerrhoC9}
|\rho(D;R-m')|\ge C_9,
\end{align} 
by $(\ref{lower_rho})$.

Under the above preparation, we shall prove inequality $(\ref{assertion gap upper bound})$. 
 For any integers $h$ and $R$ with $1\le h \le \tau$ and (\ref{assumpR}), 
we take $m'$ as above. 
Put $$R_1:=\max\{R^{\prime} \mid R^{\prime}<R, |Y_{R^{\prime}}| \ge C_{12}\}.$$ 
 Then we have $R-m'\in (i(h),i(h+1))$ and $$Y_{R-m'-1}=\dfrac{B_{D}}{\beta}\rho(D;R-m')+\dfrac{1}{\beta}Y_{R-m'}.$$

First we assume that $$|Y_{R-m'}|\ge C_{12}=\frac{|B_D|}{2|\beta|}C_9.$$ Then we have $R_1\ge R-m'$ and $R-R_1\le m'\le 2 C_{13}{\rm{log}}_{|\beta|}N$, 
 which implies $(\ref{assertion gap upper bound})$.

{In the case of $|Y_{R-m'}|< C_{12}$, using $|\beta|>1$ and  (\ref{lowerrhoC9}), we get that 
\begin{align*} 
|Y_{R-m'}| &\ge \left|\dfrac{B_{D}}{\beta}\rho(D;R-m')\right|-\left|\dfrac{1}{\beta}Y_{R-m'}\right| \\
                &\ge \dfrac{|B_D|}{|\beta|}C_9-C_{12}=C_{12}.
\end{align*}}
Therefore, we deduce $$R-R_1\le m+1\le 2C_{13}{\rm{log}}_{|\beta|}N,$$ which completes the proof of Lemma $\ref{gap upper bound}$.
\end{proof}
\subsection{Completion of the proof of Theorem $\ref{Kaneko}$}
We shall show by Lemma $\ref{gap upper bound}$ that  there exists a constant $C_{14}$ satisfying 
the following: if $N\gg 1$, then 
\begin{align}\label{last}
y_N(h)\ge \left\lfloor \dfrac{{\rm{Card}}\hspace{0.4mm}I_h}{C_{14}{\rm{log}}_{|\beta|}N} \right\rfloor,
\end{align} 
for any $1\leq h\leq \tau$. 
 In fact, take $C_{14}$ with $C_{14}>4C_{13}$. 
If ${\rm{Card}} \hspace{0.4mm} I_h\ge 4C_{13}{\rm{log}}_{|\beta|}N$, then (\ref{last}) follows from Lemma $\ref{gap upper bound}$. 
In the case of ${\rm{Card}}\hspace{0.4mm} I_h< 4C_{13}{\rm{log}}_{|\beta|}N$, 
we get (\ref{last}) because the right-hand side is equal to 0. \par
 Using relations $(\ref{=y_N})$, $(\ref{last})$, $(\ref{=N})$ and $(\ref{upper tau})$, we obtain 
\begin{align*}
y_N&=\sum_{h=1}^{\tau}y_N(h) 
     \ge \sum_{h=1}^{\tau}\left(\dfrac{{\rm{Card}}\hspace{0.4mm}I_h}{C_{14}{\rm{log}}_{|\beta|}N}-1\right)  \\
     &\ge \dfrac{N}{C_{14}{\rm{log}}_{|\beta|}N}-\tau  
     \ge \dfrac{N}{C_{14}{\rm{log}}_{|\beta|}N}-\lambda(\Gamma;N)^{D-1}. 
\end{align*}
Then by Lemma $\ref{upper bound y_N}$, we have
\begin{align*} 
C_{15}\left({\rm{log}}N+\lambda(\Gamma;N)^{D}\right)\ge y_N\ge \dfrac{N}{C_{14}{\rm{log}}_{|\beta|}N}-\lambda(\Gamma;N)^{D-1}.
\end{align*}
By the inequality above, we conclude for any $N\gg 1$ that  
\begin{align*}
(1+C_{15})\lambda(\Gamma;N)^{D}\ge \dfrac{N}{2C_{14}{\rm{log}}_{|\beta|}N},
\end{align*}
which completes the proof of Theorem $\ref{Kaneko}$.

\section*{Acknowledgements}

The authors would like to thank Professor Shigeki Akiyama for useful comments on rotational beta expansion. 
The authors are grateful to Professor Paul Surer for useful information on zeta-expansion
{{and sincerely thank the referee of the present article for significant and precise comments.}}
The first author was supported by JSPS KAKENHI Grant Numbers 15K17505 and 19K03439. 

\bibliography{}

\medskip\vglue5pt
\vskip 0pt plus 1fill
\hbox{\vbox{\hbox{Hajime \textsc{Kaneko}}
\hbox{Institute of Mathematics, University of Tsukuba, 1-1-1}
\hbox{Tennodai, Tsukuba, Ibaraki, 305-8571, JAPAN;} 
\hbox{Research Core for Mathematical Sciences} 
\hbox{University of Tsukuba, 1-1-1} 
\hbox{Tennodai, Tsukuba, Ibaraki, 305-8571, JAPAN}
\hbox{{\tt kanekoha@math.tsukuba.ac.jp}}}}
\vskip 0pt plus 1fill
\hbox{\vbox{\hbox{Makoto \textsc{Kawashima}}
\hbox{Faculty of Production Engineering, Nihon University, 2-11-1}
\hbox{Shinsakae, Narashino,}
\hbox{Chiba, 275-, 8576, Japan}
\hbox{{\tt kawashima.makoto@nihon-u.ac.jp}}
}}
\end{document}